\documentclass[12pt,english]{article}
\usepackage{babel}
\usepackage[cp1250]{inputenc}
\usepackage[T1]{fontenc}
\usepackage{amsfonts}
\usepackage{amsmath}
\usepackage{amsthm}
\usepackage{hyperref}
\usepackage{enumerate}
\usepackage{graphicx}

\newtheorem{corollary}{\textbf{Corollary}}
\newtheorem{prop}{\textbf{Proposition}}
\newtheorem{common}{\textbf{Theorem}}

\newcommand*{\esssup}{\mathop{\mathrm{ess\hspace{0.1cm}sup}}\limits}
\newcommand*{\divv}{\mathop{\mathrm{div}}\limits}
\newcommand*{\curl}{\mathop{\mathrm{curl}}\limits}

\newcommand{\te}{\theta}
\newcommand{\ut}{u_{\te}}
\newcommand{\ur}{u_{r}}
\newcommand{\uz}{u_{z}}

\newcommand{\utr}{u_{\te,r}}

\newcommand{\urr}{u_{r,r}}
\newcommand{\uzr}{u_{z,r}}

\newcommand{\utz}{u_{\te,z}}

\newcommand{\urz}{u_{r,z}}

\newcommand{\uzz}{u_{z,z}}

\newcommand{\ot}{\omega_{\te}}

\newcommand{\oor}{\omega_{r}}
\newcommand{\oz}{\omega_{z}}

\newcommand{\otr}{\omega_{\te,r}}
\newcommand{\orr}{\omega_{r,r}}
\newcommand{\ozr}{\omega_{z,r}}

\newcommand{\otz}{\omega_{\te,z}}

\newcommand{\orz}{\omega_{r,z}}

\newcommand{\ozz}{\omega_{z,z}}

\newtheorem{remark}{{\textbf {Remark}}}
\newcommand{\hd}{\hspace{0.2cm}}
\newcommand{\vd}{\vspace{0.2cm}}
\newcommand{\ep}{\varepsilon}
\newcommand{\no}{\noindent}
\newcommand{\poch}[2]{{#1}_{,#2}}

\newcommand{\eqq}[2]{\begin{equation}  #1  \label{#2} \end{equation}    }

\newcommand{\rmi}{r^{\mu}}
\newcommand{\utm}{\frac{\ut}{\rmi}}
\newcommand{\utmb}{\Big| \utm\Big|}
\newcommand{\umb}[1]{{\utmb}^{#1}}
\newcommand{\n}[1]{\left\| #1 \right\| }
\newcommand{\be}[1]{\left| #1 \right| }

\newcommand{\ra}{r^{\alpha}}
\newcommand{\ota}{\frac{\ot}{\ra}}
\newcommand{\otab}{\Big| \ota\Big|}
\newcommand{\oab}[1]{{\otab}^{#1}}

\newcommand{\jrk}{\frac{1}{r^{2}}}
\newcommand{\urpr}{\frac{ \ur}{r}}

\newcommand{\umpr}{\frac{\ur^{-}}{r}}

\newcommand{\urp}{\ur^{+}}

\newcommand{\dz}{\delta_{0}}

\newcommand{\bu}{\mathbf{u}}
\newcommand{\m}[1]{\mbox{#1}}
\newcommand{\rr}{\mathbb{R}}

\title{Remarks on Regularity Criteria for Axially Symmetric
 Weak Solutions to the Navier-Stokes Equations, II}

\author{Adam Kubica}

\begin{document}

\maketitle

\begin{center}
Faculty  of Mathematics and Information Science \\
Warsaw University of Technology \\
ul. Koszykowa 75, Warsaw 00-662 \\
a.kubica@mini.pw.edu.pl \\
\end{center}

\begin{abstract}
 We examine the conditional regularity of the solutions of Navier-Stokes equations
 in  the entire three-dimensional space under the assumption that the
 data are axially symmetric. We show that if positive part of the radial  component of velocity satisfies a weighted Serrin condition and in addition angular component
 satisfies some condition, then  the solution is regular.

\end{abstract}

\section{Introduction}

Let us consider the Navier--Stokes equations in entire
three-dimensional space \eqq{
\begin{array}{rlc}
\displaystyle \frac{\partial \bu}{ \partial t} + \bu \cdot \nabla \bu - \nu \Delta
\bu + \nabla p =\mathbf{0} & \m{ in }   & (0,T)\times \rr^{3}, \\
\divv{\bu} = 0 & \m{ in } & (0,T)\times \rr^{3}, \\
\bu(0,\mathbf{x} )= \bu_{0}(\mathbf{x}) & \m{ in } & \rr^{3} ,\\
\end{array}
}{aaa}

\no where $\bu:(0,T) \times \rr^{3} \rightarrow \rr^{3}$ is the
velocity field, $p:(0,T) \times \rr^{3} \rightarrow \rr $ is the
pressure, $0<T \leq \infty$, $ \nu$ is the viscosity coefficient,
$\bu_{0}$ is the initial velocity and the forcing term is, for the
sake of simplicity, considered to be zero. Our main result is following

\begin{common}
Let $\bu$ be a weak solution to problem (\ref{aaa}) satisfying the
energy inequality with $\bu_{0} \in W^{2,2}(\rr^{3})$ and  $r{\ut}(0) \in
L^{\infty}(\rr^{3})$. Let $\bu_{0}$ be axisymmetric. If, in
addition, $\urp$ a positive part of radial component of velocity satisfies $r^{d}\urp \in L^{w,s}((0,T) \times (\rr^{3}\cap \{r<\delta_{1} \})) $ for some $s \in (\frac{3}{2}, \infty)$, $w\in (1, \infty)$ and $d \in (-1,1)$ such that \mbox{$\frac{2}{w}+ \frac{3}{s}+ d =1$} for some positive $\delta_{1}$ and  $ r^{1-\dz} \ut \in L^{\infty}((0,T) \times \rr^{3})$ for some positive  $\dz$,
 then $(\bu,p)$, where $p$ is the corresponding
pressure,  is an axisymmetric strong solution to problem (\ref{aaa})
which is unique in the class of all weak solutions satisfying the
energy inequality. \label{thmradial}
\end{common}

\no It is convenient to write the equations (\ref{aaa})  in
cylindrical coordinates

\eqq{ \poch{\ur}{t} +\ur \urr + \uz \urz - \frac{1}{r} \ut^{2}+
\poch{ p }{r}- \nu \Big[\frac{1}{r} \poch{\big(r \urr\big)}{r}  +
\poch{\ur}{zz}- \frac{\ur}{r^{2}}\Big] =0,}{na}

\eqq{ \poch{\ut}{t} +\ur \utr + \uz \utz + \frac{1}{r} \ut\ur -\nu
\Big[\frac{1}{r} \poch{\big(r \utr\big)}{r}  + \poch{\ut}{zz}-
\frac{\ut}{r^{2}}\Big] =0,}{nb}

\eqq{ \poch{\uz}{t} +\ur \uzr + \uz \uzz + \poch{ p}{z} - \nu
\Big[\frac{1}{r} \poch{\big(r \uzr\big)}{r}  + \poch{\uz}{zz}
\Big]=0 .}{nc}

\no The equation of continuity has the following form in cylindrical
coordinates

\eqq{\urr + \frac{\ur}{r}+ \uzz=0.}{nd}

\no We put

\eqq{\mbox{\boldmath$\omega$} = \curl{\bu}.}{ne}

\vd \no We have

\[
 \oor = -\utz, \hd \ot = \urz -\uzr, \hd \oz = \utz+\frac{\ut}{r},
\]

\vd \vd

\no hence we get

\eqq{ \poch{\oor}{t} +\ur \orr + \uz \orz - \urr \oor-\urz \oz - \nu
\Big[\frac{1}{r} \poch{\big(r \orr\big)}{r}  + \poch{\oor}{zz}-
\frac{\oor}{r^{2}}\Big] =0,}{nf}

\eqq{ \poch{\ot}{t} +\ur \otr + \uz \otz - \frac{\ur}{r} \ot+2
\frac{\ut}{r} \oor - \nu \Big[\frac{1}{r} \poch{\big(r \otr
\big)}{r}  + \poch{\ot}{zz}- \frac{\ot}{r^{2}}\Big] =0,}{ng}

\eqq{ \poch{\oz}{t} +\ur \ozr + \uz \ozz -\uzr \oor- \uzz \oz- \nu
\Big[\frac{1}{r} \poch{\big(r \ozr\big)}{r}  + \poch{\oz}{zz}
\Big]=0 .}{nh}

\no Suppose that $0<t^{\ast}<T$ is the time of the first blow up of
the solutions, i.e. the smaller positive number such that
$\sup\limits_{t\in (0, t^{\ast})} \| \nabla \bu (t, \cdot )
\|_{L^{2}(\rr^{3})}=\infty$. Then, for $0<\bar{t}<t^{\ast}$ the
equations (\ref{na})-(\ref{nc}) and (\ref{nf})-(\ref{nh}) are
satisfied in $((0,\bar{t})\times \rr^{3})$ in strong sense. We will
show that it is impossible, if $\urp$ and $\ut$ satisfy our assumptions.

\vd \no First we multiply (\ref{nb}) by $\umb{p-2}\frac{\ut}{r^{2\mu}}$,
then after integrating by parts we get

\vd

\eqq{\frac{1}{p} \frac{d}{dt}
\n{\utm}^{p}_{p}+\frac{4(p-1)\nu}{p^{2}} \int \be{\nabla \utmb^{
 \frac{p}{2} } }^{2} +\nu(1-\mu^{2}) \int \utmb^{p} \frac{1}{r^{2}}+(1+\mu)\int \frac{\ur^{+}}{r}\utmb^{p}=
  (1+\mu)\int \frac{\ur^{-}}{r}\utmb^{p}.}{d}

\vd \no Next, we multiply (\ref{ng}) by
$\oab{q-2}\frac{\ot}{r^{2\alpha}}$, then after integrating by parts
we get

\[
\frac{1}{q} \frac{d}{dt} \n{\ota}^{q}_{q}+\frac{4\nu(q-1)}{q^{2}}
\int \be{\nabla \oab{\frac{q}{2} } }^{2} +\nu(1-\alpha^{2}) \int
\oab{q}\frac{1}{r^{2}} +(1-\alpha) \int \frac{\ur^{-}}{r} \oab{q}
\]

\eqq{= (1-\alpha) \int \frac{\ur^{+}}{r} \oab{q} +2\int
\frac{\ut}{r} \utz \oab{q-2} \frac{\ot}{r^{2\alpha}}.  }{i}

\vd \no Thus we have

\[
\frac{1}{p} \frac{d}{dt} \n{\utm}^{p}_{p}+\frac{1}{q} \frac{d}{dt}
\n{\ota}^{q}_{q}
+ \frac{4(p-1)\nu}{p^{2}} \int \be{\nabla \utmb^{
 \frac{p}{2} } }^{2}+\frac{4\nu(q-1)}{q^{2}}
\int \be{\nabla \oab{\frac{q}{2} } }^{2}
\]

\[
+\nu(1-\mu^{2}) \int \utmb^{p} \frac{1}{r^{2}}+\nu(1-\alpha^{2}) \int \oab{q}\frac{1}{r^{2}}
+(1+\mu)\int \frac{\ur^{+}}{r}\utmb^{p}+(1-\alpha) \int
\frac{\ur^{-}}{r} \oab{q}
\]

\eqq{
=(1+\mu)\int \frac{\ur^{-}}{r}\utmb^{p}+(1-\alpha) \int \frac{\ur^{+}}{r} \oab{q} +2\int
\frac{\ut}{r} \utz \oab{q-2} \frac{\ot}{r^{2\alpha}}\equiv (1+\mu)I_{1}+(1-\alpha)I_{2}+2I_{3}.
}{main}

\section{Estimate of $I_{3}$}

\begin{prop}
For $\gamma \in (0,3)$, \hd $q\in (\frac{2}{4-\gamma},2) $, \hd $p =\frac{(4-\gamma)q}{2}$, \hd $\mu \in (-1,1)$ and \hd $a\in (0,1)$ we have

\eqq{|I_{3}| \leq \ep_{1} \int \utmb^{p-2} \Big| \frac{\utz}{ r^{\mu}} \Big|^{2} + \ep_{2} \int \utmb^{p} \frac{1}{r^{2}} +
\ep_{3}  \int \otab^{q} \frac{1}{r^{2}} + C \int  \otab^{q},  }{aa}

\no where

\eqq{\alpha = 2\mu - \frac{\gamma}{2} (1+ \mu )- \frac{2(q-1)}{q}(1-a), }{ab}

\no and $C=C(\gamma, q, a, \ep_{1},\ep_{2}, \ep_{3})$.

\label{rem3}
\end{prop}

\begin{proof}
We have

\[
I_{3}=\int
\frac{\ut}{r} \utz \oab{q-2} \frac{\ot}{r^{2\alpha}} \leq \int \utmb^{\frac{p}{2}-1} \be{\frac{\utz}{r^{\mu}}} \cdot \frac{|\ut|^{2-\frac{p}{2}}}{r^{1+\alpha - \frac{\mu p}{2}}}
\otab^{q-1}
\]
\[
\overset{Y(2,2)}{\leq } \ep_{1} \int \utmb^{p-2} \be{ \frac{\utz}{r^{\mu}}}^{2} + C(1/\ep_{1}) \int \frac{|\ut|^{4-p}}{r^{2[1+\alpha - \frac{\mu p }{2}]}} \otab^{2(q-1)}
\]
\[
= \ep_{1} \int \utmb^{p-2} \be{ \frac{\utz}{r^{\mu}}}^{2} + C(1/\ep_{1}) \int  |r \ut |^{\gamma} \cdot
\frac{|\ut|^{4-p- \gamma}}{r^{2+2\alpha - \mu p + \gamma -\frac{4(q-1)}{q}a }}  \cdot \frac{|\ot|^{2(q-1)a} }{r^{2\alpha (q-1)a + \frac{4(q-1)}{q}a }}\cdot \frac{|\ot|^{2(q-1)(1-a)} }{r^{2\alpha (q-1)(1-a) }}.
\]

\no  Applying Young inequality with exponents $(\infty, \frac{q}{2-q}, \frac{q}{2(q-1)a}, \frac{q}{2(q-1)(1-a)})$ we get

\[
|I_{3}| \leq  \ep_{1} \int \utmb^{p-2} \be{ \frac{\utz}{r^{\mu}}}^{2} +  \ep_{2} \int \frac{|\ut|^{[4-p-\gamma]\frac{q}{2-q}} }{r^{\frac{q}{2-q}[ 2+2\alpha - \mu p + \gamma -\frac{4(q-1)}{q}a ]}}
+ \ep_{3}\int \otab^{q} \frac{1}{r^{2}} + C \int \otab^{q},
\]

\no where $C$ depends on $\ep_{1}$, $\ep_{2}$, $\ep_{3}$, $a$  and $\|r\ut  \|_{L^{\infty}}$.  But we have $[4-p-\gamma]\frac{q}{2-q}=p$ and \linebreak $\frac{q}{2-q}[ 2+2\alpha - \mu p + \gamma -\frac{4(q-1)}{q}a ]= p\mu +2$.

\end{proof}

\section{Estimate of $I_{1}$}

\begin{remark}
For all \hd $q\in (1, \infty)$, \hd $\alpha $ and $\ep_{0}$  satisfy $-2+ \ep_{0} <\alpha <\ep_{0}$. Then  there exists a constant $C=C(q,\alpha , \ep_{0})$ such that

\eqq{\int \be{\frac{\ur}{r^{1+ \alpha}}}^{q}  \cdot \frac{1}{r^{2-\ep_{0} q}} \leq C \int \otab^{q} \cdot \frac{1}{r^{2-\ep_{0} q}}.
 }{ac}

\label{rem1}
\end{remark}

\begin{proof}
We have $\| \urpr \|_{q} \leq c(q) \| \ot \|_{q}$ \hd for all $q\in (1, \infty)$. So we have to verify that $r^{-q(\alpha+ \frac{2}{q}- \ep_{0} )} $ is $A_{q}$ weight. This holds if
\[
-2< -q(\alpha+ \frac{2}{q}- \ep_{0} ) <2(q-1),
\]
\no  i.e. $-2+ \ep_{0} <\alpha <\ep_{0}$.

\end{proof}

\begin{prop}
Assume that $\varpi\equiv \| r^{1-\dz} \ut \|_{L^{\infty}}\leq C$ for some $\dz \in (0, \frac{1}{3})$. Then for all \hd  $\gamma \in (0,3)$, $q\in (\frac{2}{4-\gamma}, 2)$,
$a\in (1-\frac{(4-\gamma)q^{2}}{4(q-1)\dz},1)\cap (0,1)$,
\eqq{
\mu \in (q \dz-1,q\dz+ \frac{\gamma}{4-\gamma})\cap (-1,1),}{ae}
\no  and for $\ep_{4}$, $\ep_{5} \in (0,1)$ the following estimate holds

\eqq{|I_{1}|= \int \umpr \utmb^{p}  \leq  \ep_{4} \int \utmb^{p} \jrk +  \ep_{5} \int \otab^{q} \jrk +C \int \otab^{q},  }{ad}

\no where $p =\frac{(4-\gamma)q}{2} $, \hd $\alpha = 2\mu - \frac{\gamma}{2} (1+ \mu )- \frac{2(q-1)}{q}(1-a), $ \hd and $C=C(\ep_{4}, \ep_{5},a, q, \dz, \gamma, \varpi )$.

\label{rem2}
\end{prop}

\begin{proof}
We denote $\kappa = - \frac{2(q-1)}{q}(1-a)$. Then we can write

\[
I_{3}= \int \umpr \utmb^{p}  = \int  \be{\frac{\ut}{r^{\mu+ \frac{2}{p}}}}^{\frac{p(q-1)}{q}} \cdot  |r^{1- \dz} \ut |^{\frac{p}{q}} \cdot  \frac{\ur^{-}}{r^{1+\alpha+ \frac{2}{q}- \kappa - \frac{p}{q}\dz  }}  \overset{Y( \frac{q}{q-1},q)}{\leq }
\]
\[
\leq \ep_{4} \int \utmb^{p}  \frac{1}{r^{2}} + C(q,\gamma, \varpi, \ep_{4} ) \int \be{\frac{\ur}{r^{1+ \alpha}}}^{q} \frac{1}{r^{2-\kappa q - \dz p }} .
\]

\no We define  $\ep_{0}$  by equality \hd $-\kappa q - \dz p =- q [ \frac{(4-\gamma)q}{2} \dz - \frac{2(q-1)}{q}(1-a) ] \equiv - q \ep_{0} $. Then  using (\ref{ae})  we deduce that $-2+\ep_{0}<\alpha <\ep_{0}$,     hence we can use remark~\ref{rem1} and we get

\[
I_{3} \leq \ep_{4} \int \utmb^{p}  \frac{1}{r^{2}} + C(q,\gamma, \varpi, \ep_{4},a , \mu  ) \int \otab^{q} \frac{1}{r^{2-\ep_{0} q }}.
\]

\no Using the assumption on $a$ we get that  $b= 1- \frac{q \ep_{0} }{2}$ satisfies  $b \in (0,1)$ and we can write
\[
\int \otab^{q} \frac{1}{r^{2-\ep q }} = \int \be{\frac{\ot}{r^{\alpha + \frac{2}{q}}}}^{bq} \cdot \otab^{(1-b)q} \hd \hd  \overset{Y(\frac{1}{b}, \frac{1}{1-b})}{\leq} \hd
\ep_{5} \int \otab^{q} \frac{1}{r^{2}} + C(1/\ep_{5}) \int \otab^{q}.
\]

\no Thus we get

\[
|I_{3}| \leq \ep_{4} \int \utmb^{p}  \frac{1}{r^{2}} + \ep_{5} \int \otab^{q} \frac{1}{r^{2}}+  C \int \otab^{q} ,
\]

\no where $C=C(q,\gamma, \varpi, \ep_{4},\ep_{5},a , \mu  )$.
\end{proof}

\no From propositions~\ref{rem3} and \ref{rem2} we get

\begin{corollary}
Assume that $\varpi\equiv \| r^{1-\dz} \ut \|_{L^{\infty}}\leq C$ for some $\dz \in (0, \frac{1}{3})$. Then for all \hd  $\gamma \in (0,3)$, $q\in (\frac{2}{4-\gamma}, 2)$,
$a\in (1- \frac{(4-\gamma)q^{2}}{4(q-1)} \dz,1)$,
\eqq{
\mu \in  (q \dz-1,q\dz+ \frac{\gamma}{4-\gamma})\cap (-1,1),}{ag}
\no  and for $\ep_{1}$, $\ep_{2}$, $\ep_{3}  \in (0,1)$ the following estimate holds

\eqq{|I_{1}|+ |I_{3}| \leq   \ep_{1} \int \utmb^{p} \jrk +  \ep_{2} \int \otab^{p} \jrk + \ep_{3}
\int \be{\nabla \utmb^{ \frac{p}{2} } }^{2}+C \int \otab^{p},  }{ah}

\no where $p =\frac{(4-\gamma)q}{2} $, \hd $\alpha = 2\mu - \frac{\gamma}{2} (1+ \mu )- \frac{2(q-1)}{q}(1-a), $ \hd and $C=C(\ep_{1},\ep_{2}, \ep_{3},a, q, \dz, \gamma, \varpi )$.

\label{wn1}
\end{corollary}

\begin{corollary}
Assume that $\varpi\equiv \| r^{1-\dz} \ut \|_{L^{\infty}}<\infty$ for some $\dz \in (0, \frac{1}{3})$. Then for $\ep\in (0,\frac{1}{14})$ such that
\eqq{\Big( \frac{1-2\ep}{1-\ep} \Big)\ep \leq \dz , }{ai}

\no for  all $\ep_{1},\ep_{2},\ep_{3} \in (0,1)$ the following estimate holds

\eqq{|I_{1}|+ |I_{3}| \leq   \ep_{1} \int \utmb^{p} \jrk +  \ep_{2} \int \otab^{p} \jrk + \ep_{3}
\int \be{\nabla \utmb^{ \frac{p}{2} } }^{2}+C \int \otab^{p},  }{aj}

\no where $p=2(1-\ep^{2})$, $q=2(1-\ep)$, $\mu=\frac{1-\ep}{1+\ep}$ and $\alpha = - 2(1-2\ep)(1+\ep)\ep$ and $C=C(\ep_{1},\ep_{2}, \ep_{3},\dz,\varpi, \ep)$. In particular, for such exponents we have

\eqq{\int \be{\frac{\ur}{r^{1+ \alpha}} }^{q} \leq c(q,\alpha) \int \otab^{q},  }{al}

\no and

\eqq{\frac{2}{\infty}+ \frac{3}{q}-1-\alpha \leq \frac{1}{2}+ 7\ep<1}{ak}

\label{wn2}
\end{corollary}

\begin{proof}
We have to verify the assumptions of corollary~\ref{wn1}. Therefore we put $\gamma= 2(1-\ep)$. Then $\gamma\in (0,3)$ and $q \in (\frac{2}{4- \gamma},2)$ and we set
 $a=1-2(1- \ep^{2})\ep$. Then using (\ref{ai})  we get $a\in (1- \frac{(4-\gamma)q^{2}}{4(q-1)} \dz,1)$. Finally,  $\mu=\frac{1-\ep}{1+\ep} $ satisfies (\ref{ag}), because $\dz $
 is positive. Then we get (\ref{ah}) with $p=\frac{(4-\gamma)q}{2}= 2(1- \ep^{2})$ and $\alpha = 2\mu - \frac{\gamma}{2} (1+ \mu )- \frac{2(q-1)}{q}(1-a) = - 2(1-2\ep)(1+\ep)\ep$.

 In order to get (\ref{al}) we have to verify that $r^{-q\alpha} $ is $\mathcal{A}_{q}$ weight. Indeed, in our case we have $\frac{2}{q}-2<\alpha<\frac{2}{q}$.  Inequality (\ref{ak}) we obtain by direct calculations.

\end{proof}

\section{Estimate of $I_{2}$}

\begin{prop}
Assume that for some positive $\delta_{1}$ holds  $t \mapsto f(t)\equiv [\int_{\rr^{3} \cap \{ r<\delta_{1}\} } |r^{d}u_{r}^{+}|^{s}dx]^{\frac{w}{s}}$ is integrable for some $s \in (\frac{3}{2}, \infty)$, $w\in (1, \infty)$ and $d \in (-1,1)$ such that
$\frac{2}{w}+ \frac{3}{s}+ d =1$. Then for  $q\in (1,\infty)$ and $\alpha \in (-1,1)$ and  for all $\ep_{1}, \ep_{2}\in (0,1)$ the following estimate holds

\eqq{|I_{2}|\leq \ep_{1} \int \otab^{q} \frac{1}{r^{2}}+ \ep_{2}\int \be{\nabla \oab{\frac{q}{2} } }^{2} + C [f(t)+g(t) ] \int \otab^{q} ,  }{ba}

\no where $C=C(\ep_{1}, \ep_{2}, \delta_{1},s,w,q)$ and $g(t)=\int | \urp|^{\frac{10}{3}}$.
\label{prop2}
\end{prop}

\no Weak solutions belong to $L^{\frac{10}{3}}$, thus function $g(t)$ is integrable.

\begin{proof}
Let $\eta=\eta(r)$ be smooth cut off function such that $\eta(r)=1$ for $r< \delta_{1}/2$ and $\eta(r)=0$ for $r> \delta_{1}$. The we have
\[
I_{2}= \int \frac{\eta \ur^{+}}{r} \oab{q}+ \int \frac{(1- \eta)\ur^{+}}{r} \oab{q}\equiv I_{2,0}+ I_{2,1}.
\]

\no We begin with the first integral. We denote $a=\frac{2}{2- (\frac{2}{w}+ \frac{3}{s})} $, \hd $b= \frac{2s}{w}+3  $. Then $a>1$ and $b>3$ and we can write
\[
I_{2,0}\equiv \int \frac{\eta \ur^{+}}{r} \oab{q} = \int  \otab^{\frac{q}{a}} \frac{1}{r^{\frac{2}{a}}} \cdot \eta \urp r^{\frac{2-a}{a}} \otab^{q\frac{a-1}{a}}
\]
\[
  \overset{Y(a, \frac{a}{a-1})}{\leq } \ep_{1} \int \otab^{q} \frac{1}{r^{2}}+ c(\ep_{1},a) \int {|\eta \urp|}^{\frac{a}{a-1}} r^{\frac{2-a}{a-1}} \cdot \otab^{q} .
\]

\no Now we estimate the last integral on the right hand side

\[
\int {|\eta \urp|}^{\frac{a}{a-1}} r^{\frac{2-a}{a-1}} \cdot \otab^{q} \overset{H(\frac{b}{2}, \frac{b}{b-2})}{\leq} \Big[\int |\eta \urp|^{\frac{ab}{2(a-1)}} r^{\frac{b(2-a)}{2(a-1)}}\Big]^{\frac{2}{b}} \cdot \Big[\int \otab^{\frac{qb}{b-2}}\Big]^{\frac{b-2}{b}}
\]
\[=
\Big[\int |\eta \urp|^{\frac{ab}{2(a-1)}} r^{\frac{b(2-a)}{2(a-1)}}\Big]^{\frac{2}{b}} \cdot \Big[\int \otab^{q\frac{b-3}{b-2}} \cdot \otab^{\frac{3q}{b-2}}\Big]^{\frac{b-2}{b}}
\]
\[
\overset{H(\frac{b-2}{b-3},b-2)}{\leq}  \Big[\int |\eta \urp|^{\frac{ab}{2(a-1)}} r^{\frac{b(2-a)}{2(a-1)}}\Big]^{\frac{2}{b}} \cdot \Big[ \int \otab^{q}\Big]^{\frac{b-3}{b}} \cdot
\Big[ \int \otab^{3q}\Big]^{\frac{1}{b}}
\]
\[
\overset{Y(\frac{b}{3}, \frac{b}{b-3})}{\leq } \ep_{2} \Big[ \int \otab^{3q}\Big]^{\frac{1}{3}}+ c(\ep_{2},b) \Big[\int |\eta \urp|^{\frac{ab}{2(a-1)}} r^{\frac{b(2-a)}{2(a-1)}}\Big]^{\frac{2}{b-3}} \cdot \Big[ \int \otab^{q}\Big].
\]
\no By definition we have $\frac{ab}{2(a-1)} =s$, $\frac{b(2-a)}{2(a-1)}=ds$ and $\frac{2}{b-3}= \frac{w}{s}$, thus

\eqq{I_{2,0} \leq \ep_{1} \int \otab^{q} \frac{1}{r^{2}}+\ep_{2} \Big[ \int \otab^{3q}\Big]^{\frac{1}{3}}+ c(\ep_{1},\ep_{2},w,s)  f(t) \cdot \Big[ \int \otab^{q}\Big].}{qqa}

\no In order to estimate $I_{2,1}$ we put $a=4$ and $b=5$ and then proceeding analogously we get

\eqq{I_{2,1} \leq \ep_{1} \int \otab^{q} \frac{1}{r^{2}}+\ep_{2} \Big[ \int \otab^{3q}\Big]^{\frac{1}{3}}+ c(\ep_{1},\ep_{2})  \Big[\int |(1-\eta) \urp|^{\frac{10}{3}} r^{-\frac{5}{3} }\Big] \cdot \Big[ \int \otab^{q}\Big].}{qqb}

\no Clearly $\int |(1-\eta) \urp|^{\frac{10}{3}} r^{-\frac{5}{3} }  \leq (2/\delta_{1})^{5/3}\int | \urp|^{\frac{10}{3}} $ and the last function in integrable\footnote{Weak solutions belong to $L^{\frac{10}{3}}$. } on $(0,T)$.
Finally, applying Sobolev imbedding theorem in estimates (\ref{qqa}) and (\ref{qqb}) we get (\ref{ba}).

\end{proof}

\begin{corollary}
Assume that $\varpi\equiv \| r^{1-\dz} \ut \|_{L^{\infty}}\leq C$ for some $\dz \in (0, \frac{1}{3})$ and
\eqq{
t \mapsto f(t)\equiv \Big[\int\limits_{\rr^{3}\cap \{r<\delta_{1} \}} |r^{d}u_{r}^{+}|^{s}dx\Big]^{\frac{w}{s}}  \hd \mbox{ is integrable on } \hd (0,T)
 }{maina}
 \no for some $s \in (\frac{3}{2}, \infty)$, $w\in (1, \infty)$ and $d \in (-1,1)$ such that \mbox{$\frac{2}{w}+ \frac{3}{s}+ d =1$} and  $\delta_{1}$ positive. Then  for $\ep\in (0,\frac{1}{14})$ such that
\eqq{\Big( \frac{1-2\ep}{1-\ep} \Big)\ep \leq \dz , }{bc}

\no the following estimate holds

\[
 \frac{d}{dt} \n{\utm}^{p}_{p}+ \frac{d}{dt}
\n{\ota}^{q}_{q}
+ \frac{4(p-1)\nu}{p} \int \be{\nabla \utmb^{
 \frac{p}{2} } }^{2}+\frac{4\nu(q-1)}{q}
\int \be{\nabla \oab{\frac{q}{2} } }^{2}
\]

\eqq{
+\nu(1-\mu^{2}) \int \utmb^{p} \frac{1}{r^{2}}+\nu(1-\alpha^{2}) \int \oab{q}\frac{1}{r^{2}}
\leq C[1+f(t)+g(t)] \n{\ota}^{q}_{q},}{bf}

\no where  $p=2(1-\ep^{2})$, $q=2(1-\ep)$, $\mu=\frac{1-\ep}{1+\ep}$ and $\alpha = - 2(1-2\ep)(1+\ep)\ep$ and $C=C(\nu, \ep, \dz,\delta_{1}, \varpi,s,w)$ and $g(t)=\int | \urp|^{\frac{10}{3}}$. In particular,

\eqq{\esssup_{t\in (0,T)}{\n{\ota}_{q}} \leq C'\Big[ \n{\utm(0)}_{p}+ \n{\ota(0)}_{q} \Big], }{bd}

\no where $C'=C'(C,\|f\|_{L^{1}(0,T)})$ and $\mathbf{u}$ is regular on $(0,T)$.

\label{reg1}
\end{corollary}

\begin{proof}
Under our assumption we can use corollary~\ref{wn2} and proposition~\ref{prop2}  and we get (\ref{bf}) and by Gronwall lemma we obtain (\ref{bd}). Then using inequality  (\ref{al})
 we get $\esssup_{t\in (0,T)}{\n{\frac{\ur}{r^{1+\alpha}}}_{q}}\leq \mbox{const.}$ and (\ref{ak}) holds. Therefore we can apply  theorem~1 \cite{KPZ} and we deduce the regularity of $\mathbf{u}$.

\end{proof}

\begin{remark}
The condition (\ref{maina}) can be weakened a bit. Namely it is enough  to assume that
\eqq{
t \mapsto \widetilde{f}(t)\equiv \frac{\Big[\int |r^{d}u_{r}^{+}|^{s}dx\Big]^{\frac{w}{s}} }{1+\ln^{+}{\Big( \n{\utm}^{p}_{p}+
\n{\ota}^{q}_{q} \Big) } } \hd \mbox{ is integrable on } \hd (0,T),
 }{mainaa}

\no  where $d,w,s,p,q$ are as above. Indeed,  from (\ref{bf}) we have

\[
 \frac{d}{dt} \Big( \n{\utm}^{p}_{p}+ \n{\ota}^{q}_{q} \Big) \leq C[1+f(t)] \Big( \n{\utm}^{p}_{p}+ \n{\ota}^{q}_{q} \Big),
\]

\no so arguing similarly as in \cite{MS}  we can write

\[
\frac{d}{dt} \ln\Big[ 1+ \ln^{+}\big(\n{\utm}^{p}_{p}+ \n{\ota}^{q}_{q}\big) \Big] \leq C\frac{[1+f(t)] }{1+ \ln^{+}\Big( \n{\utm}^{p}_{p}+ \n{\ota}^{q}_{q} \Big) }.
\]

\no After integrating with respect time we obtain the bound for $\esssup_{t\in (0,T)}{\n{\ota}_{q}}$.
\end{remark}

\end{document}